\documentclass[12pt]{amsart}
\usepackage{amsmath,amssymb,amsbsy,amsfonts,amsthm,latexsym,
            amsopn,amstext,amsxtra,euscript,amscd,amsthm,graphicx,comment}
\usepackage[left=3.5cm,right=3.5cm,top=2cm,bottom=2cm,includeheadfoot]{geometry}
\usepackage[enableskew]{youngtab} 
\usepackage{lineno}
\newtheorem{lem}{Lemma}
\newtheorem{lemma}[lem]{Lemma}

\newtheorem{prop}{Proposition}
\newtheorem{proposition}[prop]{Proposition}

\newtheorem{thm}{Theorem}
\newtheorem{theorem}[thm]{Theorem}

\newtheorem*{Main}{Main Theorem}

\newtheorem{cor}{Corollary}
\newtheorem{corollary}[cor]{Corollary}

\newtheorem{defi}{Definition}
\newtheorem{definition}[defi]{Definition}

\theoremstyle{remark}
\newtheorem*{remark}{Remark}
\newtheorem*{Remarks}{\bf Remarks}

\def\\{\cr}
\def\({\left(}
\def\){\right)}
\def\<{\langle}
\def\>{\rangle}

\def\func#1{\mathop{\rm #1}}%
\newcommand{\vmu}{{\mu}}
\newcommand{\lmu}{\ell ({\mu})}
\newcommand{\smu}{\vert {\mu}\vert}

\newcommand{\vla}{{\lambda}}
\newcommand{\lla}{\ell ({\lambda})}

\def\mathscr{\mathcal}       

\begin{document}
\title[Formulas Coefficients]
{Formulas for coefficients of polynomials assigned to arithmetic functions}
\author{Bernhard Heim }
\address{Faculty of Mathematics, Computer Science, and Natural Sciences,
RWTH Aachen University, 52056 Aachen, Germany}
\email{bernhard.heim@rwth-aachen.de}
\author{Markus Neuhauser}
\address{Kutaisi International University (KIU), Youth Avenue, Turn 5/7 Kutaisi, 4600 Georgia}
\email{markus.neuhauser@kiu.edu.ge}
\subjclass[2010]{Primary 05A10, 11B83, 11P84; Secondary  11F20, 33C45}
\keywords{Arithmetic functions, Dedekind eta-function, partitions, polynomials}
\pagenumbering{arabic}

\begin{abstract}
We attach to normalized (non-vanishing)  arithmetic
functions $g$ and $h$ recursively defined polynomials.
Let $P_0^{g,h}(x):=1$. Then
\begin{equation}
P_n^{g,h}(x) := \frac{x}{h(n)} \sum_{k=1}^{n} g(k) \, P_{n-k}^{g,h}(x).
\end{equation}
For special $g$ and $h$, we obtain the D'Arcais polynomials, which are equal to
the coefficients of the $(-x)$th powers of the Dedekind $\eta$-function and are also
given by Nekrasov and Okounkov as a hook length formula. Examples are offered by
Pochhammer polynomials,
Chebyshev polynomials of the second kind, and
associated Laguerre polynomials. We present explicit formulas and identities for the coefficients of $P_n^{g,h}(x)$
which separate the impact of $g$ and $h$. Finally, we provide several applications.
\end{abstract}
\maketitle
\newpage
\section{Introduction}
The investigation of formulas and properties for the $q$-expansion of $r$th powers of Euler products
$\prod_{n=1}^{\infty }(1-q^n)^{r}$
goes back to Euler 1748 ($r=\pm 1$), Jacobi 1828 ($r=3$), 
and Ramanujan 1916 ($r=24$) \cite{Ra16}. 
This involves pentagonal numbers, partition numbers, triangle numbers,
and the Ramanujan tau-function. 
We refer to \cite{Ka78} for a historical introduction including Macdonald identities.

The topic interrelates several fields in mathematics, for example, combinatorics \cite{AE04, Wi06}, 
modular forms and number theory \cite{Se85}, representation theory of Lie algebras \cite{Ko04,We06}, and statistical mechanics \cite{NO06}.

Modular forms represented by powers of the Dedekind $\eta$-function have
captivating properties. Some of them are still conjectural, like the Lehmer conjecture
\cite{Le47} on the non-vanishing of the coefficients of discriminant function $\Delta$.
We refer to Balakrishnan, Craig and Ono \cite{BKO20} for recent results.
Let 
$$\eta(\tau):= q^{\frac{1}{24}} \prod_{n=1}^{\infty} ( 1 - q^n),$$ where $q:= e^{2 \pi i \tau}$, $\func{Im}(\tau)>0$.
Varying $r$, Newman \cite{Ne55} and Serre \cite{Se85} investigated the 
$q$-expansion
\begin{equation}\label{Newman}
\sum_{n=0}^{\infty} p_r(n) \, q^n = \prod_{n=1}^{\infty} ( 1 - q^n)^r
\end{equation} 
and obtained remarkable results. The starting point is the fact that the coefficients are polynomials in $r$.
Newman utilized (\cite{Ne55}, formula (3)), a recursion formula for the coefficients of the polynomials 
\begin{eqnarray}
p_r(n) & = & \frac{(-1)^n}{n!} \sum_{k=0}^{n-1} (-1)^k A_k(n) \, r^{n-k},\\
A_k(n) & = & \sum_{s=1}^k s! \,  \sigma(s+1)  \, \,\sum_{\lambda=1}^s \binom{\lambda -1}{s} \, A_{k-s}(\lambda-1 -s)
\end{eqnarray}
and recorded the first eleven polynomials. Recall that $\sigma(n) := \sum_{d \mid n} d$.

Let $r$ be an even, positive integer.
Serre proved that $\eta^r$ is lacunary iff $r \in \{2,4,6,8,10,14,26\}$.
Serre reduced his proof to specific properties of the first eleven polynomials (\cite{Se85}, Lemme 3). 
These polynomials belong to an interesting type of recursively defined polynomials $P_n^{g,h}(x)$ assigned
to normalized arithmetic functions $g$ and $h$. Let $h$ be non-vanishing.
Let $P_0^{g,h}(x):=1$ and 
\begin{eqnarray} 
P_n^{g,h}(x) & = &\frac{x}{h(n)} \sum_{k=1}^n g(k) \, P_{n-k}^{g,h}(x), \label{recursion}\\
P_n^{g,h}(x) & = & \frac{1}{\prod_{k=1}^n      h(k)}   \left( A_{n,n}^{g,h} x^n + \ldots + A_{n,1}^{g,h} x \right) .
\end{eqnarray}
Let $\sigma_{\ell }(n):= \sum_{d \mid n} d^{\ell }$, $\func{id}(n)=n$ and $1(n)=1$. Then $\sigma(n)= \sigma_1(n)$. We recover
Newman's approach (\ref{Newman}):
\begin{equation}
\sum_{n=0}^{\infty} P_n^{\sigma, \func{id}}(z)\,  q^n   =  \left( 1 - q^n \right)^{-z}, \qquad ( z \in \mathbb{C}).
\end{equation}
for polynomials with $(g,h)= ( \sigma, \func{id})$.  
The polynomials of degree $n \leq 5$ first appeared in
work by Francesco D'Arcais \cite{DA13}. They are called D'Arcais polynomials \cite{We06, HN20A}.

The Lehmer conjecture translates to $P_n^{\sigma, \func{id}}(-24) \neq 0$ for all $n \in \mathbb{N}$.
Let
$E_4(\tau) = 1 + 240\sum_{n=1}^{\infty} \sigma_3(n) \, q^n$ and
$E_6(\tau) = 1 - 504 \sum_{n=1}^{\infty} \sigma_5(n) \, q^n$
be the Eisenstein series of weight $4$ and $6$. They generate the algebra of modular forms for the full modular group \cite{On03}. 
Let $a_4(n)$ and $a_6(n)$ be 
the coefficients of the $q$ expansion of $1/E_4$ and $1/E_6$. Then \cite{HN20B}
\begin{eqnarray*}
a_4(n) &=& P_n^{\sigma_3, 1}(-240), \\
a_6(n)  &= &  P_n^{\sigma_5, 1}(504).
\end{eqnarray*}
In 2003, Nekrasov and Okounkov \cite{NO03} discovered a remarkable hook length formula,
displaying the $n$th
D'Arcais polynomial as a sum over all partitions $\lambda$ of $n$, ($\lambda \vdash n$),
of products involving the multiset $\mathcal{H}(\lambda)$
of hook lengths associated to $\lambda$. 
Nekrasov and Okounkov's result (finally published in 2006 \cite{NO06}) 
is based on random partitions and the Seiberg--Witten theory:
\begin{equation}\label{NO}
P_n^{\sigma,\func{id}}(z) = \sum_{\lambda \vdash n} \,\, \prod_{ h \in \mathcal{H}(\lambda)} \left( 1 + \frac{z+1}{h^2} \right).
\end{equation}
Shortly after their discovery, Westbury \cite{We06} and Han \cite{Ha10} spotted and proved the formula in connection with
Macdonald identities.
\newline

In this paper we present a formula for the coefficients $A_{n,m}^{g,h}$ of the
polynomials $P_n^{g,h}(x)$ attached to normalized arithmetic functions $g$ and $h$, where $h$ is non-vanishing.
The main goal is to identify the impact of $g$ and $h$ separately. The formula makes it possible to
study the polynomials and coefficients by varying $g$ and $h$. 
We expect to obtain an approximation of the coefficients with this method of the D'Arcais polynomials in the near future. 
\newline

\begin{Main}
Let $g$ and $h$ be normalized arithmetic functions. Let $h$ be non-vanishing.
Let $1 \leq m < n$. Let $\mu=(\mu_1, \ldots, \mu_r)$ be a partition of $n-m$.
Let $\mathcal{G}(\mu) := \prod_{k=1}^r g(\mu_k +1)$ and $\mathcal{H}(\mu,n)$ as 
provided by Definition \ref{def: Hmun} and Definition \ref{Hcal}.
Both functions depend on $\mu$. Additionally, 
$\mathcal{G}$ depends on $g$ and $\mathcal{H}(\mu,n)$ depends on $h$ and $n$.
Then we have
\begin{equation}\label{TOP}
A_{n,m}^{g,h} =   \sum_{\mu \vdash n-m}  \mathcal{G}(\mu) \cdot \mathcal{H}(\mu,n).
\end{equation}
\end{Main}

Let a partition 
$\mu=(\mu_1, \ldots, \mu_r)$ of $n$ be given. We denote by $\ell (\mu)=r$ the
length of $\mu$ and $\vert \mu \vert =n$ the size of $\mu$.
We denote $\mathfrak{m}_j = \mathfrak{m}_j(\mu)$,
the multiplicity of $j$ in the partition $\mu$.
The symmetric group $S_r$ acts
on the set of all compositions of $n$ of length $r$. Let $\func{Orb}\,(\mu)$ be
the orbit of $\mu$ \cite{St99}.
\begin{theorem}\label{th1}
Let $g$ be a normalized arithmetic function. Let $A_{n,m}^{g,1}$ be the
$m$th coefficient of $P_n^{g,1}(x)$. Let $1 \leq m <n$. Then
\begin{equation}
A_{n,m}^{g,1} =   \sum_{\mu \vdash n-m}  \mathcal{G}(\mu) \cdot 
\binom{\lmu}{\mathfrak{m}_1 , \ldots , \mathfrak{m}_{m}} \,\, \binom{n - \smu}{\lmu}.
\end{equation}
\end{theorem}
Let $1 \leq m <n$.
The coefficients for $n-m=1,2,3$ are given by
\begin{itemize}
\item  $A_{n,n-1}^{g,1}  =g\left( 2\right) \left( n-1\right) $,
\item  $A_{n,n-2}^{g,1} =\left( g\left( 2\right) \right) ^{2}\binom{n-2}{2}+g\left( 3\right) \left( n-2\right) $,
\item  $A_{n,n-3}^{g,1} =\left( g\left( 2\right) \right) ^{3}\binom{n-3}{3}+2
g\left( 2\right)
g\left( 3\right) \binom{n-3}{2}
+g\left( 4\right)
\left( n-3\right) $.
\end{itemize}

\begin{theorem}\label{th2}Let $g$ be a normalized arithmetic function. Let $A_{n,m}^{g,\func{id}}$ be the
$m$th coefficient of $P_n^{g,\func{id}}(x)$ for $1 \leq m <n$.
Then 
\begin{equation}
A_{n,m}^{g,\func{id}} =  \sum_{\mu \vdash n-m}  \mathcal{G}(\mu) \,\, 
\prod_{k=0}^{\smu + \lmu -1} (n-k) \,\,
\sum_{{\lambda} \in \func{Orb}\,(\vmu)} 
\prod_{k=1}^{\lla} \left( k + \sum_{i=1}^k \lambda_i \right)^{-1}.
\end{equation}
\end{theorem}
Let $1 \leq m <n$.
The coefficients for $n-m=1,2,3$ are given by
\begin{itemize}
\item  $A_{n,n-1}^{g,\func{id}}=g\left( 2\right) \binom{n}{2}$,

\item  $A_{n,n-2}^{g,\func{id}}=   3\left( g\left( 2\right) \right) ^{2}\binom{n}{4}+2g\left( 3\right) \binom{n}{3}$,

\item  $A_{n,n-3}^{g,\func{id}}=
15\left( g\left( 2\right) \right) ^{3}\binom{n}{6}+20g\left( 2\right) g\left( 3\right) \binom{n}{5}+6g\left( 4\right) \binom{n}{4}$.
\end{itemize}
\ \newline

Next, let us consider the most simple arithmetic function $g=1$, which involves the Pochhammer polynomials.
We have $P_n^{1,1}(x) = x (x+1)^{n-1}$ and $P_n^{1, \func{id}}(x) = \frac{1}{n!} \prod_{k=0}^{n-1} (x+ k)$. This follows directly from
Definition (\ref{recursion}). Thus,
\begin{eqnarray}
A_{n,m}^{1,1} & = & \binom{n-1}{m-1},\\
A_{n,m}^{1, \func{id}} & = & \vert s(n,m) \vert,
\end{eqnarray}
where $s(n,m)$ is the Stirling number of the first kind.
Note
that Theorem \ref{th1} and Theorem \ref{th2} already provide
non-trivial identities. This example is already quite interesting, since it shows the impact of $h$ on the root distribution for $g=1$.
Let $h$ be a normalized and non-vanishing arithmetic function. Throughout the paper we put $h(0):=0$.
Then
\begin{equation}
P_n^{1,h}(x) = \frac{1}{\prod_{k=1}^nh\left( k\right) }  \prod_{k=0}^{n-1} ( x + h(k)).
\end{equation}

\begin{remark}
It would be beneficial to develop a deformation theory for the recursively defined polynomials $P_n^{g,h}(x)$.
This should lead to new results for the distribution of the roots, where
$h$ varies from $h=1$ to $h=\func{id}$.

Let $g=1$. For $h=1$ the set of roots is given by $0$ and by $-1$ with multiplicity $n-1$. This set should {\it move towards}
the set of simple roots given by $0,1, \ldots, n-1$ for $h=\func{id}$.
The coefficients of the involved polynomials can be described in terms of elementary symmetric polynomials.
\end{remark}

Recently \cite{HN20C}, we discovered a conversion formula between the
$m$th coefficients of the $n$th polynomials assigned to 
$g$ and to $\tilde{g}(n) = g(n)/n$ for $h=1$ and $h=\func{id}$. 

Let $g$ be of moderate growth, i.~e.\ the generating series of
$g$ is regular at $0$. Let $\mathcal{G}$ be the function attached to 
$g$ and $\widetilde{\mathcal{G}}$ be the function attached to $\tilde{g}$.
Then
\begin{equation}\label{converting}
\frac{A_{n,m}^{g,\func{id}}}{n!} = \frac{A_{n,m}^{\tilde{g},1}}{m!}.
\end{equation}
This implies
\begin{corollary}
Let $g$ be a normalized arithmetic function of moderate growth. 
Let $1 \leq m < n$. Let $\mathcal{G}$ be assigned to $g$ and $\widetilde{\mathcal{G}}$ assigned to $\tilde{g}$:
\begin{eqnarray*}
\frac{1}{n!}
\sum_{\mu \vdash n-m}  \mathcal{{G}}(\mu) \,\, \prod_{k=0}^{\smu + \lmu -1} (n-k) \,\,
\sum_{{\lambda} \in \func{Orb}\,(\vmu)} 
\prod_{k=1}^{\lla} \left( k + \sum_{i=1}^k \lambda_i \right)^{-1} & & \\
=\frac{1}{m!}
\sum_{\mu \vdash n-m}  \widetilde{\mathcal{G}}(\mu) \cdot 
\binom{\lmu}{\mathfrak{m}_1 , \ldots , \mathfrak{m}_{m}} \,\, \binom{n - \smu}{\lmu} \phantom{xxxxxxx}
\end{eqnarray*}
where $\mathfrak{m}_{j}$ is the
multiplicity of $j$ in $\mu $.
\end{corollary}
\section{Deformation of the coefficients of D'Arcais polynomials}
\subsection{Notation}
Let $n \in \mathbb{N}$. A composition of $n$ is an ordered sum of integers. 
It is a sequence $\vmu = (\mu_1, \ldots, \mu_r)$ of positive integers, 
which sum up to $n :=\smu $, which is denoted by the size or weight of the
partition.
We denote by $\lmu =r$, the length of the composition and $\mathcal{C}(n)$, the set of
all compositions of $n$. Next we define the subset $\mathcal{P}(n)$ of
partitions of $n$.
A partition of $n$ is any composition $\vmu$ of $n$, where the sequence $\vmu_1 \geq \ldots \geq \vmu_r$ is non-increasing.
If $r$ summands appear in a composition of $\mu$, we say
that $\mu$ has $r$ parts. The number of all compositions and partitions of $n$ is denoted by $c(n)$ and $p(n)$, respectively.
The number of compositions and partitions with $k$ parts is denoted by $c_k(n)$ and $p_k(n)$, respectively. Recall that
$c_k(n) = \binom{n-1}{k-1}$.

The symmetric group $S_r$ acts on the set of all compositions of $n$ of length
$r$ by permuting the parts: $\pi(\vmu):= (\mu_{\pi ^{-1}(1)}, \ldots, \mu_{\pi ^{-1}(r)})$.
Each orbit 
$\func{Orb}(\vmu):= \{ \pi(\vmu) \, : \, \pi \in S_{\lmu}\}$, where $\vmu \in  \mathcal{C}(n)$ is a composition of $n$,
contains exactly one partition.
Let $\vmu \in \mathcal{C}(n)$. Then 
$\mathfrak{m} _{j}:= \mathfrak{m}_{j}(\vmu) =\left| \left\{ i:\mu _{i}=j\right\} \right| $
is the multiplicity of $j$ in the composition $\vmu$ of $n$.

Let $\varepsilon$ be defined as the unique vector of length $0$. We extend our
notation by $\ell (\varepsilon)=0$ and $\vert \varepsilon \vert = 0$. 
We define $\mathcal{C}$ as
the set of all ordered partitions, where we include the partition of length $0$.

Let $\mu_1, \ldots, \mu_r$ be non-negative integers, which add up to $n$. Then we denote by
\begin{equation}\label{multi}
\binom{n}{\vmu_1 , \ldots , \vmu_r} := \frac{n!}{\vmu_1 !
\cdots \vmu_r !}
\end{equation}
the multinomial coefficient, which has  the property 
\begin{equation}
\left( x_1 + \ldots + x_r \right)^n = 
\sum_{\substack {k_1, \ldots, k_r \geq 0 \\ k_1+ \ldots + k_r=n}} 
\binom{n}{k_1 , \ldots , k_r} \,
x_1^{k_1} \cdot x
_2^{k_2} \cdots x_r^{k_r}.
\end{equation}
Consider a multiset of $n$ objects, in which $\mu_i$ objects are of type $i$. 
Then the number of ways to linearly order these objects is equal to the 
multinomial coefficients in (\ref{multi}).

\subsection{The function $\mathcal{H}(\mu,n)$}
Throughout this section $h$ will be a normalized non-vanishing function.
Special cases are given by $h(n)=\func{id}(n)=n$ or
$h(n)=1(n)=1$ for all $n \in \mathbb{N}$.
Our goal is to define a function $\mathcal{H}(\mu,n)$ for every composition $\mu$ and every number $n$.
This function will play a significant role in our main formula for the coefficients $A_{n,m}^{g,h}$.
Let $ 1 \leq m \leq n$ be integers. We put $H(n):= \prod_{k=1}^n h(k)$ and $H(0)=1$. Further, let
$h_m(n):= \frac{H(n)}{H(n-m)} = \prod_{k=0}^{m-1} h(n-k)$.

\begin{definition} \label{def: Hmun}
Let $h$ be a normalized non-vanishing arithmetic function.
Let $\mathcal{C}$ be the set of all compositions including $\varepsilon$.
We define the function $H(\mu, n)$ on $\mathcal{C} \times \mathbb{N}$ inductively by the length of $\mu$.
The initial values $H(\varepsilon, n):=1$ of length $0$ for $n\in \mathbb{N}_0$.
Let further $\vmu \in \mathcal{C}$ of length $r+1$, where $r \geq 0$.
For $ n \geq \vert \vmu \vert + \ell (\vmu)$, let 
\begin{equation}
H(\mu,n) :=
\sum_{k = \vert \vmu \vert + \ell ( \vmu)-1}^{n-1}
h_{\mu_{r+1}}(k) \,   H( (\mu_1, \ldots, \mu_r),
k - \mu_{r+1}),
\end{equation}
If $ n < \vert \vmu \vert + \ell (\vmu)$ we put $H(\vmu, n) :=0$.
\end{definition}
\ \newline

\begin{Remarks} 
a) Let the values of $h$
be positive integers. Then function $H(\mu,n)$ is an
integer-valued function and has positive values for $ n \geq \vert \vmu \vert + l(\vmu)$. It is otherwise vanishing.
\newline
b) Let $\mu \in \mathbb{N}$ and let $n \geq \mu +1$. Then
\begin{equation} \label{scalar}
H(\mu,n) =  \sum_{k=\mu}^{n-1} h_{\mu}(k).
\end{equation}
\end{Remarks}

For $h(n)=1$, or $h(n)=\func{id}(n)$, and $\mu$ a partition of length $r$,
we have a closed formula for $H(\mu,n)$.
\begin{proposition} Let $h$ be a normalized non-vanishing arithmetic function. Let $h(n)=1$ or $h(n)=n$, then we
have the following closed formulas. Let $n \in \mathbb{N}_0$ and $n \geq 1$.
Let $\mu$ be a composition of length $\ell (\mu)$ and size $\vert \mu \vert$.
Let $h(n)=1$ then
\begin{equation}\label{eins: H}
 H(\vmu,n) = \binom{n - \vert {\mu} \vert}{\ell ({\mu})}.
\end{equation}
Let $h(n)=n$ then
\begin{equation}\label{zwei: H}
H(\vmu ,n) = 
\prod_{k=0}^{\smu + \lmu -1} (n-k) \,\,
\prod_{k=1}^{\ell \left( \mu \right) } \left( k + \sum_{j=1}^k \mu _{j} \right)^{-1}.
\end{equation}
\end{proposition}
\begin{proof} Let $h=1$ be trivial. We prove the formula (\ref{eins: H}) by induction over the length $r$. Let $n$ be arbitrary.
We start with $r=0$. Then the formula confirms that $H(\varepsilon, n)=1$.
Now let $r\geq 1$ and
(\ref{eins: H}) be true for all compositions with length smaller than $r$. Let a composition $\mu$ be given of length $r$. Then
\begin{eqnarray*}
H(\mu, n)             &=&           \sum _{k=\left| \mu \right| +\ell \left( \mu \right) -1}^{n-1}
h_{\mu _{r}}  \left( k\right) 
H((\mu _{1},\ldots ,\mu _{r -1}),
k-\mu _{r}) \\
&=& \sum _{k=\left| \mu \right| +\ell \left( \mu \right) -1}^{n-1}
\binom{k-\left| \mu \right| }{\ell \left( \mu \right) -1}=\binom{n-\left| \mu \right| }{\ell \left( \mu \right) }.
\end{eqnarray*}
The last step is given by the sum of the entries in the appropriate column in the Pascal triangle.
Formula (\ref{zwei: H}) is slightly more complicated but proven in the same way.
\end{proof}
For example, for $h \in \{ 1, \func{id} \}$ we have the following values, which illustrate some patterns and differences.
\begin{itemize}
\item
Let $h(n)=1$. \newline
Then we have $H(1,n)= n-1$, $H((1,1),n)= \binom{n-2}{2}$, $H(2,n)=n-2$, and $H((1,1,1),n)= \binom{n-3}{3}$.
Further $H((1,2),n) = H((2,1), n) = \binom{n-3}{2}$.
\item
Let $h(n)=n$. \newline
Then we have $H(1,n)= \binom{n}{2}$, $H((1,1),n) = 3 \, \binom{n}{4}$ and $H(2,n)= 2 \, \binom{n}{3}$.
Further, $H((1,2),n)= 12 \, \binom{n}{5}$ and $H((2,1), n) = 8 \, \binom{n}{5}$.
\end{itemize}

\begin{definition}\label{Hcal}
Let $h$ be a normalized non-vanishing arithmetic function. Let
$\vmu \in \mathcal{P}(m)$. Then 
\begin{equation}
\mathcal{H}(\vmu ,n) := \sum_{ \vla \in \func{Orb}\, (\vmu)} H(\vla ,n).
\end{equation}
\end{definition}

\subsection{Proof of Theorem \ref{th1} and Theorem \ref{th2}}
\ \newline
Theorem \ref{th1} follows from the Main Theorem, formula (\ref{eins: H}), and the following
proposition.
\begin{proposition}\label{formula:calH h=1}
Let $h=1$. Let $n$ and $m$ be positive integers and let $\vmu \in \mathcal{P}(m)$.
Let $\mathfrak{m}_{j}=\mathfrak{m}_{j}(\vmu)$
be the multiplicity of $j$ in the partition $\vmu$.
Then we have
\begin{equation}
\mathcal{H}(\vmu ,n) = \binom{\lmu}{\mathfrak{m}_1 , \ldots , \mathfrak{m}_{m}} \,\, \binom{n - \smu}{\lmu}.
\end{equation}
\end{proposition}
\begin{proof}
Let $r=\ell \left( \mu \right) $.
In the case of $h\left( n\right) =1$ for all $n$, 
$H(\pi \left( \mu \right) , n ) =H(\mu , n ) $
for every $\pi \in S_{r}$, as
$\ell \left( \pi \left( \mu \right) \right) =\ell \left( \mu \right) $
and
$\left| \pi \left( \mu \right) \right| =\left| \mu \right| $.
This implies that
$\mathcal{H}(\mu ,n) =\left| \func{Orb} \, \left( \mu \right) \right| H( \mu , n ) $.
Let
$F\left( \mu \right) =\left\{ \pi \in S_{r}:\pi \left( \mu \right) =\mu \right\} $
then
$\left| \func{Orb}\, \left( \mu \right) \right| =\frac{r!}{\left| F\left( \mu \right) \right| }$.
For $\pi \in F\left( \mu \right) $ and
$\nu =\pi \left( \mu \right) $ it must hold that
$\nu _{j}=\mu _{j}$ for all $1\leq j\leq r$ and any
$\pi $ which satisfies this property is in
$F\left( \mu \right) $.
This means that
$F\left( \mu \right) \cong \prod _{j=1}^{m}S_{\mathfrak{m}_{j}}$,
where $S_{0}$ is a trivial group consisting of only the
neutral element.
That means that
$\left| F\left( \mu \right) \right| =\prod _{j=1}^{m}\mathfrak{m}_{j}!$,
which proves the claim.
\end{proof}
Theorem \ref{th2} follows from the Main Theorem, formula (\ref{zwei: H}), and the following
proposition.
\begin{proposition}
Let $h=\func{id}$. Let $n$ and $m$ be positive integers.
Let $\vmu \in \mathcal{P}(m)$.
Then we have
\begin{equation}
\mathcal{H}\left( \vmu,n\right) = 
\prod_{k=0}^{\smu + \lmu -1} (n-k) \,\,
\sum_{{\lambda} \in \func{Orb}\,(\vmu)} 
\prod_{k=1}^{\lla} \left( k + \sum_{i=1}^k \lambda_i \right)^{-1}
\end{equation}
\end{proposition}
\section{Main Theorem}
\begin{definition}
Let $g$ be a normalized arithmetic function. Let $\vmu \in \mathcal{P}(m)$. 
The function $\mathcal{G}(\vmu)$ is defined by the product of $g$ evaluated at $(\vmu_k +1)_k$.
\begin{equation}
\mathcal{G}(\vmu) := \prod_{k=1}^{\ell (\vmu)} g(\mu_k + 1).
\end{equation}
\end{definition}

We would like to establish a formula for $A_{n,m}^{g,h}$, separating the contribution of the arithmetic functions
$g$ and $h$. As a first step, we utilize the recursive definition of $P_n^{g,h}(x)$ given in (\ref{recursion}).
For simplification of the notation, we frequently put $A_{n,m}= A_{n,m}^{g,h}$.
Recall that $A_{0,0}=1$, $A_{n,0}=0$ for $n \in \mathbb{N}$ and, 
$A_{n,n}= 1$.
\begin{lemma}\label{firststep}
The coefficients of $A_{n,m}$ of $P_{n}^{g,h}(x)$ satisfy the recursion formula
\begin{equation}
A_{n,m} = \sum_{k=1}^{n-m+1} g(k) \, \frac{H(n-1)}{H(n-k)} \,\, A_{n-k,m-1},  \qquad (1 \leq m \leq n).
\end{equation}
\end{lemma}
\begin{proof} We directly apply the definition of $P_n^{g,h}(x)$ by
the recursion stated in (\ref{recursion}).
Let $1 \leq m \leq n$. Then $\sum_{m=1}^n A_{n,m} \, x^m = H(n) \, P_n(x)$ is equal to
\begin{eqnarray*}
H\left( n-1\right) x\sum _{k=1}^{n}g\left( k\right) P_{n-k}\left( x\right)
& = & x\sum _{k=1}^{n}g\left( k\right) \frac{H\left( n-1\right) }{H\left( n-k\right) }\sum _{m=0}^{n-k}A_{n-k,m}x^{m} \\
&=&\sum _{m=1}^{n}\sum _{k=1}^{n-m+1}g\left( k\right) \frac{H\left( n-1\right) }{H\left( n-k\right) }
A_{n-k,m-1}x^{m}.
\end{eqnarray*}
\end{proof}
The value of $A_{n,m}$ is determined by the $n-m+1$ values
$$A_{n-1,m-1}, A_{n-2,m-1}, \ldots, A_{m-1, m-1}.$$
These are all $m-1$st
coefficients of all polynomials $P_d(x)$ of degree $0 \leq d <n$.
Let $\delta := n-m$. From the computational point of view, to calculate $A_{n,m}$, we 
need $\delta +1 $ many previously given coefficients. As a special case we have
\begin{equation}
A_{n,n}=\frac{H\left( n-1\right) }{H\left( n-1\right) }A_{n-1,n-1}=A_{0,0}=1.
\end{equation}
\subsection{Proof of the Main Theorem}
\begin{proof}
Let $\delta:= n-m >0$.
The proof is given by induction on the degree $n$ of the polynomials $P_n(x)$. Let $n \in \mathbb{N}$.
To formalize the induction, let $S(n)$
be the mathematical statement: 
\begin{equation}\label{statement:n}
S(n): \,\,\, A_{n, n - \delta} \text{ satisfies } (\ref{TOP}) \text{ for all } 0 < \delta \leq n.
\end{equation}
The statement $S(1)$ is true, since $A_{n,0}=0$ for all $n \in \mathbb{N}$ and $A_{n,0}=0$ for $\delta=n=1$, 
evaluating the formula (\ref{TOP}). This follows from $\mathcal{H}(1,1)=0$.

Let us now assume that the statements $S(1), \ldots, S(n-1)$ are true. We show that this implies, that $S(n)$ is also true.
Our starting point is provided by the following formula, stated in Lemma \ref{firststep}:
\begin{equation}
A_{n,n-\delta} = A_{n-1, n -\delta -1} + \sum_{k=2}^{\delta +1} g(k) \, 
\frac{H(n-1)}{H(n-k)} \,\, A_{n-k,n- (\delta+1)},  \qquad (0  \leq \delta < n).
\label{iteration}
\end{equation}
First, we replace $k$ by $k+1$ and
utilize the fact that the last formula (\ref{iteration})
is also true for $n$ substituted by $n-1, n-2, \ldots, \delta +1$:
\[
A_{n,n-\delta} = \sum _{N=\delta }^{n-1} \sum_{k=1}^{\delta } g(k+1) \,
\frac{H(N)}{H(N-k)} \,\, A_{N-k,N- \delta}.
\]
Next we insert the induction hypothesis.
For $A_{n,n- \delta}$ we obtain the expression
\begin{eqnarray*}
& & \sum _{N=\delta}^{n-1}
\sum _{k=1}^{\delta}\sum _{\mu \in \mathscr{P}\left( \delta -k\right) }g\left( k+1\right) 
\prod _{j=1}^{\ell \left( \mu \right) }g\left( \mu _{j}+1\right) h_{k}\left( N\right) \mathcal{H}(\mu ,N-k) \\
&=&
\sum _{k=1}^{\delta}\sum _{\mu \in \mathscr{P}\left( \delta-k\right) }g\left( k+1\right) 
\prod _{j=1}^{\ell \left( \mu \right) }g\left( \mu _{j}+1\right) \sum _{N=\delta}^{n-1}h_{k}\left( N\right) \mathcal{H}(\mu ,N-k).
\end{eqnarray*}
Let $r=\ell \left( \mu \right) $ and $M=\left| \mu \right| +k$, then
we use the one to one correspondence that
$\nu =\left( \nu _{1},\ldots ,\nu _{r},\nu _{r+1}\right) \in \mathbb{N}^{r+1}$ 
if and only if
$\mu = \left( \nu _{1},\ldots ,\nu _{r}\right) \in \mathscr{C} \left( M-\nu _{r+1}\right) $.
Recall that in every orbit of a composition, there is always exactly one partition. 
This allows us finally to obtain for $A_{n,n- \delta}$ the expression
\begin{eqnarray*}
& &\sum _{\nu \in \mathscr{P}\left( \delta \right) }g\left( \nu _{r+1}
+1\right) \prod _{j=1}^{\ell \left( \nu \right) -1}g\left( \nu _{j}+1\right) \sum _{M=\nu _{r+1}}^{n-1}h_{
\nu _{r+1}}\left( M\right) \mathcal{H}((\nu _{1},\ldots ,\nu _{r}), M-
\nu _{r+1}) \\
&=&\sum _{\nu \in \mathscr{P}\left( \delta \right) }
\prod _{j=1}^{\ell \left( \nu \right) }g\left( \nu _{j}+1\right)
\mathcal{H}\left( \nu ,n\right)
\end{eqnarray*}
which shows the claim.
\end{proof}
\section{Applications and examples}
We begin with some definitions and results.
Let us first recall some properties of symmetric polynomials.
Let $R$ be a commutative ring and let $t_1,\ldots,t_n$ be algebraic independent elements of $R$.
Let $x$ be a variable over $R[t_1, \ldots,t_n]$. We expand the polynomials
\[
p(x) = \prod_{k=1}^{n} \left( x - t_k \right)
= x^n - s_1 x^{n-1} + \ldots + (-1)^n s_n,
\]
where each
$s_{k}=s_k (t_1, \ldots , t_n)$ is a polynomial in $t_1, \ldots, t_n$.
For instance,
\begin{equation}
s_1 = \sum_{k=1}^n t_k, \qquad s_2 = \sum_{ 1 \leq k_1 < k_2 \leq n} t_{k_1} t_{k_2}, \qquad \ldots , \qquad s_n = \prod_{k=1}^n t_k.
\end{equation}
The polynomials $s_1, \ldots, s_n$ are the elementary symmetric polynomials of $t_1, \ldots, t_n$.
The symmetric group $S_n$ operates on 
$R[t_1, \ldots,t_n]$ by $$\pi( q(t_1, \ldots, t_n)) := q(t_{\pi ^{-1}(1)}, \ldots, t_{\pi ^{-1}(n)}),
\qquad \text{where }q \in R[t_1, \ldots,t_n], \, \pi \in S_n.$$
A polynomial $q$ is symmetric iff $\pi(q)=q$ for all $\pi \in S_n$.
It is well-known that every symmetric polynomial $q(t_1, \ldots, t_n)$ is a polynomial in
the elementary symmetric polynomials.
We further recall the definitions of unimodal, log-concave,
and ultra-log-concave.
\begin{definition}
Let $a_{0},a_{1},\ldots ,a_n$ be a finite sequence of
non-negative real numbers. This sequence is called:
\begin{itemize}
\item    \emph{unimodal} if
$a_{0} \leq a_{1} \leq \ldots \leq a_{k-1} \leq a_{k} \geq a_{k+1} \geq \ldots \geq  a_{n}$
for some $k$,

\item   \emph{log-concave} if
$a_{j}^2 \geq a_{j-1}a_{j+1}$ for all $j\geq 1$, and

\item   \emph{ultra-log-concave} if the
associated sequence $a_k/ \binom{n}{k}$ is log-concave.
\end{itemize}
\end{definition}
Note that ultra-log-concave implies log-concave and
log-concave implies unimodal. The converses are in
general not true, indicated by the following examples.
The sequence of coefficients of the polynomials:
\begin{itemize}
\item  
$x^{2}+2x+5$ is unimodal but not
log-concave and
\item  $x^{2}+2x+3$ is log-concave but not
ultra-log-concave.
\end{itemize}
Let $P_n^{\tilde{g},1}(x)$ be ultra-log-concave or log-concave, then
$P_n^{g,\func{id}}(x)$ is ultra-log-concave or log-concave, respectively (see \cite{HN20C}).
This follows from relation~(\ref{converting}) and
$$ \frac{1}{m!^{2}}-\frac{1}{\left( m-1\right) !\left( m+1\right) !}=\frac{1}{\left( m-1\right) !m!}\left( \frac{1}{m}-\frac{1}{m+1}\right) >0.$$
Let $\tilde{g}\left( n\right) =g\left( n\right) /n$. Then
$\left( A_{n,m}^{\tilde{g},1}\right) ^{2}\geq A_{n,m-1}^{\tilde{g},1}A_{n,m+1}^{\tilde{g},1}$
implies
\begin{eqnarray*}
\left( A_{n,m}^{g,\func{id}}\right) ^{2}&=&n!^{2}\left( A_{n,m}^{\tilde{g},1}\right) ^{2}\frac{1}{m!^{2}}\\
&\geq &n!^{2}A_{n,m-1}^{\tilde{g},1}A_{n,m+1}^{\tilde{g},1}\frac{1}{\left( m-1\right) !\left( m+1\right) !}=A_{n,m-1}^{g,\func{id}}A_{n,m+1}^{g,\func{id}}.
\end{eqnarray*}
The proof for ultra-log-concave is exactly the same
if the coefficients
$A_{n,k}$ are replaced by
$A_{n,k} / \binom{n}{k}$.
\newpage
\subsection{The case $g=1$}
\subsubsection{Arbitrary $h$} \ \newline
Let $h$ be a normalized non-vanishing arithmetic function. Then $P_n^{1,h}(x)$ is determined by
\begin{eqnarray*}
P_n^{1,h}(x) & = & \frac{1}{H(n)} (x + h(0))\cdots \left( x+ h(n-1)\right) \\
& = & \frac{1}{H(n)}   \sum_{m=1}^n   s_{n-m}(h(0), \ldots, h(n-1)) \,\, x^m.
\end{eqnarray*}
\subsubsection{Let $g=h=1$}
We have $P_n^{1,1}(x) = x (x+1)^{n-1}$ for $n \geq 1$. This
leads to $ 1 \leq m \leq n$:
\begin{equation}
A_{n,m}^{1,1} = \binom{n-1}{m-1} = s_{n-m}(0,1, \ldots, 1).
\end{equation}
Therefore, $P_n^{1,1}(x)$ is ultra-log-concave. 
Combining this formula with Theorem \ref{th1} leads to
\begin{corollary}Let $1 \leq m < n$. Then
\begin{equation}
\sum_{ \vmu \in \mathcal{P}(n-m)} \, \binom{\lmu}{\mathfrak{m} _1 , \ldots , \mathfrak{m} _{n}} \,\, \binom{n - \smu}{\lmu}
=\binom{n-1}{m-1} .
\end{equation}
\end{corollary}
\subsection{The case $g=\func{id}$} 
\subsubsection{Let $g= \func{id}$ and $h=\func{id}$}
\ \newline
The explicit form of the coefficients and the ultra-log-concave properties
immediately leads to
\begin{corollary} The polynomials $P_n^{\func{id},\func{id}}(x)$ are ultra-log-concave and the coefficients are equal to the Lah numbers.
These polynomials are associated Laguerre polynomials.
\begin{equation}
\frac{1}{n!} A_{n,m}^{\func{id},\func{id}} = \frac{1}{m!} \binom{n-1}{m-1}. 
\end{equation}
\end{corollary}
\subsubsection{Let $g= \func{id}$ and $h$ be arbitrary}
We have $g\left( n+2\right) -2g\left( n+1\right) +g\left( n\right) =0$ for all
$n\geq 1$, see \cite[Example~2.7~(1)]{HNT20}.
From \cite[Theorem~2.1]{HNT20} we obtain
\begin{equation} \label{orthogonal}
\frac{h\left( n\right) }{h\left( n+2\right) }P_{n}^{g,h}\left( x\right) +\left( -2 \,\, 
\frac{h\left( n+1\right) }{h\left( n+2\right) }-\frac{x}{h\left( n+2\right) }\right) 
P_{n+1}^{g,h}\left( x\right) +P_{n+2}^{g,h}\left( x\right) =0.
\end{equation}
We would like to mention
that in the case $g=\func{id}$ and $h=1$, we obtain the Chebyshev polynomials of the second kind.
It is not a coincidence that $P_n^{1, \func{id}}(x)$ and $P_n^{\func{id}, \func{id}}(x)$ both satisfy 
a $2$nd
order linear difference equation ($3$-term linear difference equation), since they are orthogonal polynomials.

\begin{remark} Let $g$ be
fixed. Let the assumptions of Theorem 2.1 \cite{HNT20} be fulfilled. 
Then $\left\{P_n^{g,h}(x)\right\}_n$ satisfies for all non-vanishing normalized arithmetic functions $h$
the same type of reduced recursion formula, with the same of amount of terms (see (\ref{orthogonal}) for $g = \func{id}$).
\end{remark}

\subsection{Log-concavity}
\begin{proposition}\label{log}
Let $g$ and $h$ be normalized arithmetic functions with positive values.
Let 
\begin{equation}
a_{n,m} = \frac{A_{n,m}^{g,h}}{\prod_{k=1}^n h(k)}
\end{equation}
be the $m$th coefficient of the polynomial $P_n^{g,h}(x)$. Then
\begin{equation} \label{log: n-1}
a_{n,n-1}^2 > a_{n,n-2} \,\, a_{n,n}
\end{equation}
for all $n \geq 2$ if $g(2)^2 > g(3)$. There exists $g$, independent of $h$,
such that (\ref{log: n-1}) fails.
\end{proposition}
\begin{proof}
It is sufficient to consider $ \Delta(n):= (\prod_{k=1}^n h(k))^2 \left( (a_{n,n-1}^2 - a_{n,n-2} \, a_{n,n} \right)$.
Then 
\begin{eqnarray*}
\Delta(n) & = & \left( \mathcal{G}(1) \mathcal{H}\left( 1,n\right)  \right)^2 - 
\left[
\mathcal{G}(2) \mathcal{H}\left( 2,n\right) + \mathcal{G}(1,1) \mathcal{H}\left( (1,1), n\right) \right] \\
& = & g(2)^2 \sum_{ k_1,k_2 =1}^{n-1} h(k_1) \, h(k_2) \\
& & {}-g(3) 
\sum_{k=2}^{n-1} h(k)\, h(k-1) 
- g(2)^2 \sum_{k=3}^{n-1} h(k) \sum_{\ell =1}^{k-2} h(\ell )\\
& \geq &
\left( g(2)^2-g(3) \right)
\sum_{k=2}^{n-1} h(k)\, h(k-1) .
\end{eqnarray*}
This proves the first part of the claim in the Proposition.
We also observe that we can chose $g(3)$ sufficiently large that a $n$ exists, such
that $\Delta(n)<0$.
\end{proof}
The Nekrasov--Okounkov polynomials \cite{NO06, HZ20, HN20C} defined by 
\begin{equation} 
\label{no}
Q_n(x): = \sum_{\lambda \vdash n} \,\, \prod_{ h \in \mathcal{H}(\lambda)} \left( 1 + \frac{x}{h^2} \right) = \sum_{m=0}^{n} b_{n,m} \, x^m.
\end{equation}
are log-concave for $n\leq 1500$. Hong and Zhang \cite{HZ20} have proven that for 
$n$ sufficiently large, and $m \gg \sqrt{n} \,\, \log \, n$: $b_{n,m} \geq b_{n, m+1}$.
From Proposition \ref{log} and \cite{Br94} we obtain 
\begin{corollary}
Let $n \geq 2$ and $Q_n(x)$ be the Nekrasov-Okounkov polynomial. Then
\begin{equation}
b_{n,n-1}^2 > b_{n,n-2} \, b_{n,n}.
\end{equation}
\end{corollary}

\subsection{Applications with the conversion formula}
It is well known (see for example \cite{Ko04}) that
\begin{equation}
\prod_{n=1}^{\infty} \left(1 - q^n \right)^{-x} = 1 + \sum_{n=1}^{\infty} 
\left( \sum_{m=1}^n 
\frac{1}{m!}\,
\sum_{\substack {k_1, \ldots, k_m \in \mathbb{N} \\ k_1+ \ldots + k_m=n}} \, \prod_{i=1}^{m} \frac{\sigma (k_i)}{k_i} \, x^m \right)
q^n   
\end{equation}
This can be easily generalized to a formula for the coefficients of $P_n^{g, \func{id}}(x)$.
Utilizing the conversion formula (\ref{converting}) we also obtain a formula for the coefficients
of $P_n^{g, 1}(x)$.
\begin{corollary} \label{general}
Let $g$ be a normalized arithmetic function of moderate growth. Then we obtain for
the normalized coefficients of $P_n^{g, \func{id}}(x)$
and $P_n^{g, 1}(x)$ the following formulas:
\begin{eqnarray}
\frac{A_{n,m}^{g,\func{id}}}{n!}& = & \frac{1}{m!}
\sum_{\substack {k_1, \ldots, k_m \in \mathbb{N} \\ k_1+ \ldots + k_m=n}} \, \prod_{i=1}^{m} \frac{g (k_i)}{k_i}, \\
A_{n,m}^{g,1 } & = & 
\sum_{\substack {k_1, \ldots, k_m \in \mathbb{N} \\ k_1+ \ldots + k_m=n}} \, \prod_{i=1}^{m} g (k_i).
\end{eqnarray}
\end{corollary} \
We utilized $$\sum_{n=0}^{\infty} P_n^{g, \func{id}}(x) \, q^n = \exp 
\left( x \sum_{n=1}^{\infty} g(n) \, \frac{q^n}{n}\right)$$
and the conversion formula
\begin{equation}
\frac{A_{n,m}^{g,\func{id}}}{n!} = \frac{A_{n,m}^{\tilde{g},1}}{m!}.
\end{equation}
It is a very challenging task to find similar formulas involving general $h$. 
In this paper we provide a formula (Main Theorem) for all pairs $g$ and $h$, which indicates
already the complexity of this task.

\section*{Acknowledgments}
To be entered later.

\begin{thebibliography}{MMM99}



\bibitem[AE04]{AE04} G. E. Andrews, K. Eriksson: \emph{Integer Partitions.\/} Cambridge University Press (2004).

\bibitem[Br94]{Br94}
F. Brenti: \emph{Log-concave and unimodal sequence in algebra, combinatorics and
geometry: an update.}  Jerusalem combinatorics '93,
Contemp.\ Math.\ {\bf 178}, Amer. Math. Soc., Providence, RI
(1994), 71--89.

\bibitem[BKO20]{BKO20} J. Balakrishnan, W. Craig, K. Ono:
\emph{Variations of Lehmer's Conjecture for Ramanujan's tau-function.} 
To appear in: J. Number Theory (JNT Prime and Special Issue on Modular Forms and
Function Fields) and
arXiv: https://arxiv.org/abs/2005.10345. 








\bibitem[DA13]{DA13} F. D'Arcais: 
\emph{D\'{e}veloppement en s\'{e}rie.} Interm\'{e}diaire Math. 
\textbf{20} (1913), 233--234.










\bibitem[Ha10]{Ha10} G. Han: 
\emph{The Nekrasov--Okounkov hook length formula: refinement, elementary proof and applications.} 
Ann. Inst. Fourier (Grenoble) 
\textbf{60} 
no.~1 (2010), 1--29.














\bibitem[HN20A]{HN20A}  B. Heim, M. Neuhauser: \emph{The Dedekind eta 
function and D'Arcais-type polynomials.}
Res.\ Math.\ Sci.\ 7: 3
\texttt{doi:10.1007/s40687-019-0201-5}.






















\bibitem[HN20B]{HN20B} B. Heim, M. Neuhauser: \emph{On the reciprocals of Eisenstein series.}
International Journal of number theory (12.10.2020 online: https://doi.org/10.1142/S1793042120400199)
\bibitem[HN20C]{HN20C} B. Heim, M. Neuhauser: \emph{Horizontal and vertical log-concavity.} Submitted.

\bibitem[HNT20]{HNT20}
B. Heim, M. Neuhauser, R. Tr\"{o}ger: \emph{Zeros of recursively defined polynomials.}
\newblock J. Difference Equ. Appl. \textbf{26} no.\ 4 (2020), 510--531.

\bibitem[HZ20]{HZ20}  L. Hong, S. Zhang: \emph{Towards Heim and Neuhauser's unimodality conjecture on the 
Nekrasov--Okounkov polynomials.\/}
arXiv:2008.10069.

\bibitem[Ka78]{Ka78} V. G. Kac:
\emph{Infinite-dimensional algebras, Dedekind's $\eta$-function, classical M{\"o}bius function and the very strange formula.}
Advances in Mathematics \textbf{30} (1978), 85--136.



\bibitem[Ko04]{Ko04} B. Kostant: \emph{Powers of the Euler product and commutative subalgebras of a complex simple Lie algebra.} Invent. Math.
 \textbf{158} (2004), 181--226.





\bibitem[Le47]{Le47} D. Lehmer: \emph{The vanishing of Ramanujan's $\tau(n)$.} Duke Math.\ J.
 \textbf{14} (1947), 429--433.


\bibitem[NO03]{NO03} N. Nekrasov, A. Okounkov: \emph{Seiberg--Witten theory and random partitions.} arXiv:hep-hep/0306238v2. 


\bibitem[NO06]{NO06} N. Nekrasov, A. Okounkov: \emph{Seiberg--Witten theory and random partitions.}
{\it The unity of mathematics.\/}  Progr. Math. {\bf{244}} Birkh\"{a}user Boston (2006), 525--596.

\bibitem[Ne55]{Ne55} M. Newman: \emph{An identity for the coefficients of certain modular forms\/}.
J. London Math.\ Soc.\ 
{\bf{30}} (1955), 488--493.



\bibitem[On03]{On03} K. Ono: \emph{The Web of Modularity: Arithmetic of the Coefficients of Modular Forms and q-series.}
Conference Board of Mathematical Sciences 
{\bf{102}} (2003).




\bibitem[Ra16]{Ra16} 
S. Ramanujan: \emph{On certain arithmetical
functions.} Trans.\ Cambridge Philos.\
Soc.\ \textbf{22}, 159--184 (1916). 
In: Hardy, G.~H., Seshu Aiyar, P.~V.,
Wilson, B.~M. (eds.) Collected
Papers of Srinivasa Ramanujan, pp.\
136--162. AMS Chelsea Publishing,
American Mathematical Society,
Providence,
RI (2000)








\bibitem[Se85]{Se85} 
J. Serre:
\emph{Sur la lacunarit\'{e} des puissances de $\eta $\/}.
Glasgow Math.\ J. 
{\bf{27}} (1985), 203--221.

\bibitem[St99]{St99} R. Stanley: \emph{ Enumerative  Combinatorics, vol. 1.}
Cambridge: Cambridge University Press, 1999.




\bibitem[We06]{We06} B. Westbury: \emph{Universal characters from the Macdonald identities\/}.
Adv. Math.\ 
{\bf{202}} 
no.\ 1 
(2006), 50--63.





\bibitem[Wi06]{Wi06} H. S. Wilf: \emph{Generatingfunctionology.} A. K. Peters, Wellesley, Massachusettes, third edition (2006).

 \end{thebibliography}
\end{document}